\documentclass{article}
\usepackage{amsmath}
\usepackage{amssymb}
\usepackage{amsfonts}
\usepackage{hyperref}

\newtheorem{theorem}{Theorem}

\newtheorem{corollary}[theorem]{Corollary}

\newtheorem{lemma}[theorem]{Lemma}

\newtheorem{proposition}[theorem]{Proposition}
\newtheorem{remark}[theorem]{Remark}

\newenvironment{proof}[1][Proof]{\noindent\textbf{#1.} }{\ \rule{0.5em}{0.5em}}
\input{tcilatex}
\numberwithin{theorem}{section}
\numberwithin{equation}{section}

\begin{document}

\title{Jet Riemann-Hamilton geometrization for the conformal deformed
Berwald-Mo\'{o}r quartic metric depending on momenta}
\author{Alexandru Oan\u{a} and Mircea Neagu}
\date{}
\maketitle

\begin{abstract}
In this paper we expose on the dual $1$-jet space $J^{1\ast }(\mathbb{R}%
,M^{4})$ the distinguished (d-) Riemannian geometry (in the sense of
d-connection, d-torsions, d-curvatures and some gravitational-like and
electromagnetic-like geometrical models) for the $(t,x)$-conformal deformed
Berwald-Mo\'{o}r Hamiltonian metric of order four.
\end{abstract}



\section{Introduction}

The geometric-physical Lagrangian or Hamiltonian Berwald-Mo\'{o}r structure (%
\cite{Berwald}, \cite{Moor-1}) was intensively investigated by P.K.
Rashevski (\cite{Rashevski}) and further fundamented and developed by D.G.
Pavlov, G.I. Garas'ko and S.S. Kokarev (\cite{Pavlov-1}, \cite{Garasko-1}, 
\cite{Pavlov-Kokarev}). In their works, the preceding Russian scientists
emphasize the importance in the theory of space-time structure, gravitation
and electromagnetism of the geometry produced by the classical Berwald-Mo%
\'{o}r Lagrangian metric%
\begin{equation*}
F:TM\rightarrow \mathbb{R},\mathbb{\qquad }F(y)=\sqrt[n]{y^{1}y^{2}...y^{n}},%
\mathbb{\qquad }n\geq 2,
\end{equation*}%
or by the corresponding Berwald-Mo\'{o}r Hamiltonian metric%
\begin{equation*}
H:T^{\ast }M\rightarrow \mathbb{R},\mathbb{\qquad }H(p)=\sqrt[n]{%
p_{1}p_{2}...p_{n}}.
\end{equation*}%
In such a perspective, according to the recent geometric-physical ideas
proposed by Garas'ko (\cite{Garasko-1}), we consider that a distinguished
Riemannian geometry (in the sense of d-connection, d-torsions, d-curvatures
and some gravitational-like and electromagnetic-like geometrical models) for
the conformal deformations of the jet Berwald-Mo\'{o}r Hamiltonian metric of
order four is required. Note that a similar geometric-physical study for the 
$(t,x)$-conformal deformations of the jet Berwald-Mo\'{o}r Lagrangian metric
of order four is now completely developed in the paper \cite{Neagu-x-B-M(4)}%
. Also, few elements of distinguished Hamiltonian geometry produced by the
cotangent quartic Berwald-Mo\'{o}r metric depending of momenta are already
presented in the paper \cite{At-Bal-Neagu}.

In such a geometrical and physical context, this paper investigates on the
dual$\ 1$-jet space $J^{1\ast }(\mathbb{R},M^{4})$ the Riemann-Hamilton
distinguished geometry (together with a theoretical-geometric field-like
theory) for the $(t,x)$\textit{-conformal deformed Berwald-Mo\'{o}r
Hamiltonian metric of order four}\footnote{%
We assume that we have $p_{1}^{1}p_{2}^{1}p_{3}^{1}p_{4}^{1}>0$. This is one
domain where we can $p$-differentiate the Hamiltonian function $H(t,x,p)$.}%
\begin{equation}
H(t,x,p)=2e^{-\sigma (x)}\sqrt{h_{11}(t)}\left[
p_{1}^{1}p_{2}^{1}p_{3}^{1}p_{4}^{1}\right] ^{1/4},  \label{rheon-B-M}
\end{equation}%
where $\sigma (x)$ is a smooth non-constant function on $M^{4}$, $h_{11}(t)$
is a Riemannian metric on $\mathbb{R}$, and $%
(t,x,p)=(t,x^{1},x^{2},x^{3},x^{4},p_{1}^{1},p_{2}^{1},p_{3}^{1},p_{4}^{1})$
are the coordinates of the\ momentum phase space $J^{1\ast }(\mathbb{R}%
,M^{4})$; these transform by the rules (the Einstein convention of summation
is assumed everywhere):%
\begin{equation}
\widetilde{t}=\widetilde{t}(t),\quad \widetilde{x}^{i}=\widetilde{x}%
^{i}(x^{j}),\quad \widetilde{p}_{i}^{1}=\dfrac{\partial x^{j}}{\partial 
\widetilde{x}^{i}}\dfrac{d\widetilde{t}}{dt}p_{j}^{1},  \label{tr-rules}
\end{equation}%
where $i,j=\overline{1,4},$ rank $(\partial \widetilde{x}^{i}/\partial
x^{j})=4$ and $d\widetilde{t}/dt\neq 0$. It is important to note that, based
on the geometrical ideas promoted by Miron, Hrimiuc, Shimada and Sab\u{a}u
in the classical Hamiltonian geometry of cotangent bundles (\cite%
{Miro+Hrim+Shim+Saba}), together with those used by Atanasiu, Neagu and Oan%
\u{a} in the geometry of dual $1$-jet spaces, the differential geometry (in
the sense of d-connections, d-torsions, d-curvatures, abstract
gravitational-like and electromagnetic-like geometrical theories) produced
by a jet Hamiltonian function $H:J^{1\ast }(\mathbb{R},M^{n})\rightarrow 
\mathbb{R}$ is now completely done in the papers \cite{Atan-Neag2}, \cite%
{Atan+Neag1}, \cite{Oana+Neag-2} and \cite{Oana+Neag-3}. In what follows, we
apply the general geometrical results from \cite{Oana+Neag-2} and \cite%
{Oana+Neag-3} to the square of Hamiltonian metric (\ref{rheon-B-M}), which
is given by ($n=4$)%
\begin{equation}
\overset{\ast }{H}(t,x,p)=H^{2}(t,x,p)=4e^{-2\sigma (x)}h_{11}(t)\left[
p_{1}^{1}p_{2}^{1}p_{3}^{1}p_{4}^{1}\right] ^{1/2}.  \label{rheon-B-M2}
\end{equation}

\begin{remark}
The momentum Hamiltonian metric (\ref{rheon-B-M2}) is exactly the natural
Hamiltonian attached to the square of the jet Berwald-Mo\'{o}r Lagrangian
metric$,$ which has the expression%
\begin{equation}
\overset{\ast }{L}(t,x,y)=e^{2\sigma (x)}h^{11}(t)\left[
y_{1}^{1}y_{1}^{2}y_{1}^{3}y_{1}^{4}\right] ^{1/2}.
\label{square-Lagrangian}
\end{equation}%
In other words$,$ we have%
\begin{equation*}
\overset{\ast }{H}(t,x,p)=p_{r}^{1}y_{1}^{r}-\overset{\ast }{L}(t,x,y),
\end{equation*}%
where $p_{i}^{1}=\partial \overset{\ast }{L}/\partial y_{1}^{i}$. Note that
the jet Lagrangian metric (\ref{square-Lagrangian}) is even the square of
the conformal deformed jet quartic Berwald-Mo\'{o}r Finslerian metric%
\begin{equation*}
F(t,x,y)=e^{\sigma (x)}\sqrt{h^{11}(t)}\left[
y_{1}^{1}y_{1}^{2}y_{1}^{3}y_{1}^{4}\right] ^{1/4}.
\end{equation*}%
In other words$,$ we have $\overset{\ast }{L}=F^{2}$.
\end{remark}

\section{The canonical nonlinear connection}

Using the notation $\mathcal{P}^{1111}:=p_{1}^{1}p_{2}^{1}p_{3}^{1}p_{4}^{1}$
and taking into account that we have%
\begin{equation*}
\frac{\partial \mathcal{P}^{1111}}{\partial p_{i}^{1}}=\frac{\mathcal{P}%
^{1111}}{p_{i}^{1}},
\end{equation*}%
then the \textit{fundamental metrical d-tensor} produced by the metric (\ref%
{rheon-B-M2}) is given by the formula (no sum by $i$ or $j$)%
\begin{equation}  \label{g*-jos-(ij)}
\overset{\ast }{g}\;^{\!\!ij}(t,x,p)=\frac{h^{11}(t)}{2}\frac{\partial ^{2}%
\overset{\ast }{H}}{\partial p_{i}^{1}\partial p_{j}^{1}}=\frac{e^{-2\sigma
(x)}\left( 1-2\delta ^{ij}\right) }{2}\frac{\left[ \mathcal{P}^{1111}\right]
^{1/2}}{p_{i}^{1}p_{j}^{1}}\text{.}
\end{equation}%
Moreover, the matrix $\overset{\ast }{g}=(\overset{\ast }{g}\;^{\!\!ij})$
admits the inverse $\overset{\ast }{g}\;^{\!\!-1}=(\overset{\ast }{g}_{jk})$%
, whose entries are given by%
\begin{equation}
\overset{\ast }{g}_{jk}=\frac{e^{2\sigma (x)}\left( 1-2\delta _{jk}\right) %
\left[ \mathcal{P}^{1111}\right] ^{-1/2}}{2}p_{j}^{1}p_{k}^{1}\text{ (no sum
by }j\text{ or }k\text{).}  \label{g*-sus-(jk)}
\end{equation}

Let us consider the Christoffel symbol of the Riemannian metric $h_{11}(t)$
on $\mathbb{R}$, which is given by%
\begin{equation*}
\text{\textsc{k}}_{11}^{1}=\frac{h^{11}}{2}\frac{dh_{11}}{dt},
\end{equation*}%
where $h^{11}=1/h_{11}>0$. Then, using the notation $\sigma _{i}:=\partial
\sigma /\partial x^{i}$, we find the following geometrical result:

\begin{proposition}
For the $(t,x)$-conformal deformed Berwald-Mo\'{o}r Hamiltonian metric of
order four (\ref{rheon-B-M2})$,$ the \textbf{canonical nonlinear connection}
on the dual $1$-jet space $J^{1\ast }(\mathbb{R},M^{4})$ has the components%
\begin{equation}
N=\left( \underset{1}{N}\text{{}}_{(i)1}^{(1)}=\text{\textsc{k}}%
_{11}^{1}p_{i}^{1},\text{ }\underset{2}{N}\text{{}}_{(i)j}^{(1)}=-4\sigma
_{i}p_{i}^{1}\delta _{ij}\right) .  \label{nlc-B-M}
\end{equation}
\end{proposition}

\begin{proof}
The canonical nonlinear connection produced by $\overset{\ast }{H}$ on the
dual $1$-jet space $J^{1\ast }(\mathbb{R},M^{4})$ has the following
components (see \cite{Atan-Neag2}): $\underset{1}{N}${}$_{(i)1}^{(1)}\overset%
{def}{=}$\textsc{k}$_{11}^{1}p_{i}^{1}$ and%
\begin{equation*}
\underset{2}{N}\text{{}}_{(i)j}^{(1)}=\dfrac{h^{11}}{4}\left[ \dfrac{%
\partial \overset{\ast }{g}_{ij}}{\partial x^{k}}\dfrac{\partial \overset{%
\ast }{H}}{\partial p_{k}^{1}}-\dfrac{\partial \overset{\ast }{g}_{ij}}{%
\partial p_{k}^{1}}\dfrac{\partial \overset{\ast }{H}}{\partial x^{k}}+%
\overset{\ast }{g}_{ik}\dfrac{\partial ^{2}\overset{\ast }{H}}{\partial
x^{j}\partial p_{k}^{1}}+\overset{\ast }{g}_{jk}\dfrac{\partial ^{2}\overset{%
\ast }{H}}{\partial x^{i}\partial p_{k}^{1}}\right] .
\end{equation*}%
Now, by a direct calculation, we obtain (\ref{nlc-B-M}).
\end{proof}

\section{The Cartan canonical $N$-linear connection. Its d-torsions and
d-curvatures}

\hspace{5mm}The nonlinear connection (\ref{nlc-B-M}) produces the dual 
\textit{adapted bases} of d-vector fields (no sum by $i$) 
\begin{equation}
\left\{ \frac{\delta }{\delta t}=\frac{\partial }{\partial t}-\text{\textsc{k%
}}_{11}^{1}p_{r}^{1}\frac{\partial }{\partial p_{r}^{1}}\text{ };\text{ }%
\frac{\delta }{\delta x^{i}}=\frac{\partial }{\partial x^{i}}+4\sigma
_{i}p_{i}^{1}\frac{\partial }{\partial p_{i}^{1}}\text{ };\text{ }\dfrac{%
\partial }{\partial p_{i}^{1}}\right\} \subset \mathcal{X}(E^{\ast })
\label{a-b-v}
\end{equation}%
and d-covector fields (no sum by $i$)%
\begin{equation}
\left\{ dt\text{ };\text{ }dx^{i}\text{ };\text{ }\delta
p_{i}^{1}=dp_{i}^{1}+\text{\textsc{k}}_{11}^{1}p_{i}^{1}dt-4\sigma
_{i}p_{i}^{1}dx^{i}\right\} \subset \mathcal{X}^{\ast }(E^{\ast }),
\label{a-b-co}
\end{equation}%
where $E^{\ast }=J^{1\ast }(\mathbb{R},M^{4})$. The naturalness of the
geometrical adapted bases (\ref{a-b-v}) and (\ref{a-b-co}) is coming from
the fact that, via a transformation of coordinates (\ref{tr-rules}), their
elements transform as the classical tensors. Therefore, the description of
all subsequent geometrical objects on the\ dual $1$-jet space $J^{1\ast }(%
\mathbb{R},M^{4})$ (e.g., the Cartan canonical linear connection, its
torsion and curvature etc.) will be done in local adapted components.
Consequently, by direct computations, we obtain the following geometrical
result:

\begin{proposition}
The Cartan canonical $N$-linear connection produced by the $(t,x)$-conformal
deformed Berwald-Mo\'{o}r Hamiltonian metric of order four (\ref{rheon-B-M2}%
) has the following adapted local components (no sum by $i,$ $j$ or $k$):%
\begin{equation}
C\Gamma (N)=\left( \text{\textsc{k}}_{11}^{1},\text{ }A_{j1}^{i}=0,\text{ }%
H_{jk}^{i}=4\delta _{j}^{i}\delta _{k}^{i}\sigma _{i},\text{ }%
C_{i(1)}^{j(k)}=\mathsf{C}_{i}^{jk}\cdot \frac{p_{i}^{1}}{p_{j}^{1}p_{k}^{1}}%
\right) ,  \label{Cartan-can-connect}
\end{equation}%
where%
\begin{eqnarray*}
\mathsf{C}_{i}^{jk} &=&\frac{1-2\delta ^{jk}-2\delta _{i}^{j}-2\delta
_{i}^{k}+8\delta _{i}^{j}\delta _{i}^{k}}{8}= \\
&=&\left\{ 
\begin{array}{ll}
\dfrac{1}{8}, & i\neq j\neq k\neq i\medskip \\ 
-\dfrac{1}{8}, & i=j\neq k\text{ or }i=k\neq j\text{ or }j=k\neq i\medskip
\\ 
\dfrac{3}{8}, & i=j=k.%
\end{array}%
\right.
\end{eqnarray*}
\end{proposition}

\begin{proof}
The adapted components of the Cartan canonical connection are given by the
formulas (see \cite{Oana+Neag-2})${,\quad }$%
\begin{equation*}
{A_{j1}^{i}\overset{def}{=}{\dfrac{\overset{\ast }{g}\;^{\!\!il}}{2}}{\dfrac{%
\delta \overset{\ast }{g}_{lj}}{\delta t}=0}},\quad H{_{jk}^{i}}\overset{def}%
{=}\dfrac{\overset{\ast }{g}\;^{\!\!ir}}{2}\left( {\dfrac{\delta \overset{%
\ast }{g}_{jr}}{\delta x^{k}}}+{\dfrac{\delta \overset{\ast }{g}_{kr}}{%
\delta x^{j}}}-{\dfrac{\delta \overset{\ast }{g}_{jk}}{\delta x^{r}}}\right)
=4\delta _{j}^{i}\delta _{k}^{i}\sigma _{i},
\end{equation*}%
\begin{equation*}
C_{i(1)}^{j(k)}{\overset{def}{=}-{\dfrac{\overset{\ast }{g}_{ir}}{2}}\left( {%
\dfrac{\partial \overset{\ast }{g}\;^{\!\!jr}}{\partial p_{k}^{1}}}+{\dfrac{%
\partial \overset{\ast }{g}\;^{\!\!kr}}{\partial p_{j}^{1}}}-{\dfrac{%
\partial \overset{\ast }{g}\;^{\!\!jk}}{\partial p_{r}^{1}}}\right) =-{%
\dfrac{\overset{\ast }{g}_{ir}}{2}\dfrac{\partial \overset{\ast }{g}%
\;^{\!\!jr}}{\partial p_{k}^{1}}}}.
\end{equation*}

Using the derivative operators (\ref{a-b-v}), the direct calculations lead
us to the required results.
\end{proof}

\begin{remark}
It is important to note that the vertical d-tensor $C_{j(k)}^{i(1)}$ also
has the properties (sum by $m$):%
\begin{equation}
C_{i(1)}^{j(k)}=C_{i(1)}^{k(j)},\quad C_{i(1)}^{j(m)}p_{m}^{1}=0,\quad
C_{m(1)}^{j(m)}=0,\quad C_{i(1)|m}^{j(m)}=0,  \label{equalitie-C}
\end{equation}%
where%
\begin{equation*}
C_{i(1)|j}^{l(k)}{\overset{def}{=}}\frac{\delta {C_{i(1)}^{l(k)}}}{\delta
x^{j}}+{C_{i(1)}^{r(k)}H_{rj}^{l}}-{C_{r(1)}^{l(k)}H_{ij}^{r}}+{%
C_{i(1)}^{l(r)}H_{rj}^{k}.}
\end{equation*}
\end{remark}

\begin{proposition}
The Cartan canonical connection of the $(t,x)$-conformal deformed Berwald-Mo%
\'{o}r Hamiltonian metric of order four (\ref{rheon-B-M2}) has \textbf{two}
effective local torsion d-tensors:%
\begin{equation*}
R_{(r)ij}^{(1)}=-4\sigma _{ij}\left( p_{i}^{1}\delta _{ir}-p_{j}^{1}\delta
_{jr}\right) ,\quad P_{i(1)}^{r(j)}=\mathsf{C}_{i}^{rj}\cdot \frac{p_{i}^{1}%
}{p_{r}^{1}p_{j}^{1}},
\end{equation*}%
where $\sigma _{ij}:=\dfrac{\partial ^{2}\sigma }{\partial x^{i}\partial
x^{j}}.$
\end{proposition}

\begin{proof}
A Cartan canonical connection on the dual $1$-jet space $J^{1\ast }(\mathbb{R%
},M^{4})$ generally has \textit{six} effective local d-tensors of torsion
(for more details, see \cite{Oana+Neag-2}). For the particular Cartan
canonical connection (\ref{Cartan-can-connect}) these reduce only to \textit{%
two} (the other four are zero):%
\begin{equation*}
R_{(r)ij}^{(1)}\overset{def}{=}{\dfrac{\delta \underset{2}{N}\text{{}}%
_{(r)i}^{(1)}}{\delta x^{j}}}-{\dfrac{\delta \underset{2}{N}\text{{}}%
_{(r)j}^{(1)}}{\delta x^{i}}},\quad P_{i(1)}^{r(j)}\overset{def}{=}%
C_{i(1)}^{r(j)}.
\end{equation*}
\end{proof}

\begin{proposition}
The Cartan canonical connection of the $(t,x)$-conformal deformed Berwald-Mo%
\'{o}r Hamiltonian metric of order four (\ref{rheon-B-M2}) has \textbf{three}
effective local curvature d-tensors:%
\begin{equation*}
{R_{ijk}^{l}={\dfrac{\partial H_{ij}^{l}}{\partial x^{k}}}-{\dfrac{\partial
H_{ik}^{l}}{\partial x^{j}}}%
+H_{ij}^{r}H_{rk}^{l}-H_{ik}^{r}H_{rj}^{l}+C_{i(1)}^{l(r)}R_{(r)jk}^{(1)},%
\quad }P_{ij(1)}^{l\text{ }(k)}={-C_{i(1)|j}^{l(k)},}
\end{equation*}%
\begin{equation*}
S_{i(1)(1)}^{l(j)(k)}\overset{def}{=}{{\dfrac{\partial C_{i(1)}^{l(j)}}{%
\partial p_{k}^{1}}}-{\dfrac{\partial C_{i(1)}^{l(k)}}{\partial p_{j}^{1}}}%
+C_{i(1)}^{r(j)}C_{r(1)}^{l(k)}-C_{i(1)}^{r(k)}C_{r(1)}^{l(j)}.}
\end{equation*}
\end{proposition}

\begin{proof}
A Cartan canonical connection on the dual $1$-jet space $J^{1\ast }(\mathbb{R%
},M^{4})$ generally has \textit{five} effective local d-tensors of curvature
(for all details, see \cite{Oana+Neag-2}). For the particular Cartan
canonical connection (\ref{Cartan-can-connect}) these reduce only to \textit{%
three} (the other two are zero). These are $S_{i(1)(1)}^{l(j)(k)}$ and%
\begin{equation*}
\begin{array}{l}
\medskip {R_{ijk}^{l}\overset{def}{=}{\dfrac{\delta H_{ij}^{l}}{\delta x^{k}}%
}-{\dfrac{\delta H_{ik}^{l}}{\delta x^{j}}}%
+H_{ij}^{r}H_{rk}^{l}-H_{ik}^{r}H_{rj}^{l}+C_{i(1)}^{l(r)}R_{(r)jk}^{(1)},}
\\ 
{P_{ij(1)}^{l\;\;(k)}\overset{def}{=}{\dfrac{\partial H_{ij}^{l}}{\partial
p_{k}^{1}}}-C_{i(1)|j}^{l(k)}+C_{i(1)}^{l(r)}P_{(r)j(1)}^{(1)\;%
\;(k)}=-C_{i(1)|j}^{l(k)},}%
\end{array}%
\end{equation*}%
where%
\begin{equation*}
P_{\left( r\right) j\left( 1\right) }^{\left( 1\right) \ \left( k\right) }={%
\dfrac{\partial \underset{2}{N}\overset{\left( 1\right) }{_{\left( r\right)
j}}}{\partial p_{k}^{1}}+H_{rj}^{k}=0.}
\end{equation*}
\end{proof}

\section{From $(t,x)$-conformal deformations of the quartic Berwald-Mo\'{o}r
Hamiltonian metric to field-like geometrical models}

\subsection{Momentum gravitational-like geometrical model}

\hspace{5mm}The $(t,x)$-conformal deformed Berwald-Mo\'{o}r Hamiltonian
metric of order four (\ref{rheon-B-M2}) produces on the momentum phase space 
$J^{1\ast }(\mathbb{R},M^{4})$ the adapted metrical d-tensor (sum by $i$ and 
$j$)%
\begin{equation}
\mathbb{G}=h_{11}dt\otimes dt+\overset{\ast }{g}_{ij}dx^{i}\otimes
dx^{j}+h_{11}\overset{\ast }{g}\;^{\!\!ij}\delta p_{i}^{1}\otimes \delta
p_{j}^{1},  \label{gravit-pot-B-M}
\end{equation}%
where $\overset{\ast }{g}_{jk}$ and $\overset{\ast }{g}\;^{\!\!ij}$ are
given by (\ref{g*-sus-(jk)}) and (\ref{g*-jos-(ij)}), and we have (no sum by 
$i$) $\delta p_{i}^{1}=dp_{i}^{1}+$\textsc{k}$_{11}^{1}p_{i}^{1}dt-4\sigma
_{i}p_{i}^{1}dx^{i}$. We believe that, from a physical point of view, the
metrical d-tensor (\ref{gravit-pot-B-M}) may be regarded as a \textit{%
\textquotedblleft gravitational potential depending on
momenta\textquotedblright }. In our abstract geometric-physical approach,
one postulates that the momentum gravitational potential $\mathbb{G}$ is
governed by the \textit{geometrical Einstein equations}%
\begin{equation}
\text{Ric }(C\Gamma (N))-\frac{\text{Sc }(C\Gamma (N))}{2}\mathbb{G=}%
\mathcal{K}\mathbb{T},  \label{Einstein-eq-global}
\end{equation}%
where

\begin{itemize}
\item Ric $(C\Gamma (N))$ is the \textit{Ricci d-tensor} associated to the
Cartan canonical linear connection (\ref{Cartan-can-connect}); the Cartan
canonical linear connection plays in our geometric-physical theory the same
role as the Levi-Civita connection in the classical Riemannian theory of
gravity;

\item Sc $(C\Gamma (N))$ is the \textit{scalar curvature};

\item $\mathcal{K}$ is the \textit{Einstein constant} and $\mathbb{T}$ is an
intrinsic \textit{momentum stress-energy d-tensor of matter.}\medskip
\end{itemize}

Therefore, using the adapted basis of vector fields (\ref{a-b-v}), we can
locally describe the global geometrical Einstein equations (\ref%
{Einstein-eq-global}). Consequently, some direct computations lead to:

\begin{lemma}
The Ricci tensor of the Cartan canonical connection of the $(t,x)$-conformal
deformed Berwald-Mo\'{o}r Hamiltonian metric of order four (\ref{rheon-B-M2}%
) has the following \textbf{two} effective local Ricci d-tensors (no sum by $%
i$, $j$, $k$ or $l$):%
\begin{equation}
\begin{array}{l}
\bigskip R_{ij}=\left\{ 
\begin{array}{ll}
-2\sigma _{ij}-\dfrac{p_{i}^{1}}{p_{k}^{1}}\sigma _{jk}-\dfrac{p_{i}^{1}}{%
p_{l}^{1}}\sigma _{jl}, & i\neq j,\;\;\{i,j,k,l\}=\{1,2,3,4\}\medskip \\ 
0, & i=j,%
\end{array}%
\right. \\ 
R_{(1)(1)}^{(i)(j)}=-S_{(1)(1)}^{(i)(j)}=\dfrac{4\delta ^{ij}-1}{8}\dfrac{1}{%
p_{i}^{1}p_{j}^{1}}{.}%
\end{array}
\label{Ricci-local}
\end{equation}
\end{lemma}

\begin{proof}
Generally, the Ricci tensor of a Cartan canonical connection $C\Gamma (N)$
(produced by an arbitrary\ momentum Hamiltonian function) is determined by 
\textit{six} effective local Ricci d-tensors (for more details, see \cite%
{Oana+Neag-3}). For the particular Cartan canonical connection (\ref%
{Cartan-can-connect}) these reduce only to \textit{two} (the other four are
zero), where (sum by $r$ and $m$):%
\begin{equation*}
\begin{array}{lll}
\medskip R_{ij} & \overset{def}{=} & R_{ijm}^{m}={{\dfrac{\partial H_{ij}^{m}%
}{\partial x^{m}}}-{\dfrac{\partial H_{im}^{m}}{\partial x^{j}}}%
+H_{ij}^{r}H_{rm}^{m}-H_{im}^{r}H_{rj}^{m}+C_{i(1)}^{m(r)}R_{(r)jm}^{(1)},}
\\ 
\medskip S_{(1)(1)}^{(i)(j)} & \overset{def}{=} & S_{m\left( 1\right) \left(
1\right) }^{i\left( j\right) \left( m\right) }=\dfrac{\partial C_{m\left(
1\right) }^{i\left( j\right) }}{\partial p_{m}^{1}}-\dfrac{\partial
C_{m\left( 1\right) }^{i\left( m\right) }}{\partial p_{j}^{1}}+C_{m\left(
1\right) }^{r\left( j\right) }C_{r\left( 1\right) }^{i\left( m\right)
}-C_{m\left( 1\right) }^{r\left( m\right) }C_{r\left( 1\right) }^{i\left(
j\right) }= \\ 
& = & \dfrac{\partial C_{m\left( 1\right) }^{i\left( j\right) }}{\partial
p_{m}^{1}}+C_{m\left( 1\right) }^{r\left( j\right) }C_{r\left( 1\right)
}^{i\left( m\right) }.%
\end{array}%
\end{equation*}
\end{proof}

\begin{lemma}
The scalar curvature of the Cartan canonical connection of the $(t,x)$%
-conformal deformed Berwald-Mo\'{o}r Hamiltonian metric of order four (\ref%
{rheon-B-M2}) has the value%
\begin{equation*}
\text{\emph{Sc} }(C\Gamma (N))=-4e^{-2\sigma }\left[ \mathcal{P}^{1111}%
\right] ^{1/2}\Sigma _{11}-\frac{3}{2}h^{11}e^{2\sigma }\left[ \mathcal{P}%
^{1111}\right] ^{-1/2},
\end{equation*}%
where%
\begin{equation*}
\Sigma _{11}=\dfrac{\sigma _{12}}{p_{1}^{1}p_{2}^{1}}+\dfrac{\sigma _{13}}{%
p_{1}^{1}p_{3}^{1}}+\dfrac{\sigma _{14}}{p_{1}^{1}p_{4}^{1}}+\dfrac{\sigma
_{23}}{p_{2}^{1}p_{3}^{1}}+\dfrac{\sigma _{24}}{p_{2}^{1}p_{4}^{1}}+\dfrac{%
\sigma _{34}}{p_{3}^{1}p_{4}^{1}}.
\end{equation*}
\end{lemma}

\begin{proof}
The scalar curvature of the Cartan canonical connection (\ref%
{Cartan-can-connect}) is given by the general formula%
\begin{equation*}
\text{Sc }(C\Gamma (N))=\overset{\ast }{g}\;^{\!\!ij}R_{ij}-h^{11}\overset{%
\ast }{g}_{ij}S_{(1)(1)}^{(i)(j)}.
\end{equation*}
\end{proof}

The local description in the adapted basis of vector fields (\ref{a-b-v}) of
the global geometrical Einstein equations (\ref{Einstein-eq-global}) leads
us to

\begin{proposition}
The \textbf{geometrical Einstein-like equations} produced by the $(t,x)$%
-conformal deformed Berwald-Mo\'{o}r Hamiltonian metric of order four (\ref%
{rheon-B-M2}) are locally described by%
\begin{equation}
\left\{ 
\begin{array}{l}
2e^{-2\sigma }h_{11}\left[ \mathcal{P}^{1111}\right] ^{1/2}\Sigma _{11}+%
\dfrac{3}{4}e^{2\sigma }\left[ \mathcal{P}^{1111}\right] ^{-1/2}=\mathcal{K}%
\mathbb{T}_{11},\medskip \\ 
\medskip R_{ij}+\left\{ 2e^{-2\sigma }\left[ \mathcal{P}^{1111}\right]
^{1/2}\Sigma _{11}+\dfrac{3}{4}h^{11}e^{2\sigma }\left[ \mathcal{P}^{1111}%
\right] ^{-1/2}\right\} \overset{\ast }{g}_{ij}=\mathcal{K}\mathbb{T}_{ij},
\\ 
-S_{(1)(1)}^{(i)(j)}+\left\{ \frac{{}}{{}}2e^{-2\sigma }h_{11}\left[ 
\mathcal{P}^{1111}\right] ^{1/2}\Sigma _{11}+\right. \medskip \\ 
\hspace{25mm}\left. +\dfrac{3}{4}e^{2\sigma }\left[ \mathcal{P}^{1111}\right]
^{-1/2}\right\} \overset{\ast }{g}\;^{\!\!ij}=\mathcal{K}\mathbb{T}%
_{(1)(1)}^{(i)(j)}, \\ 
\begin{array}{lll}
0=\mathbb{T}_{1i}, & 0=\mathbb{T}_{i1}, & 0=\mathbb{T}_{(1)1}^{(i)}\medskip ,
\\ 
0=\mathbb{T}_{1(1)}^{\text{ \ }(i)}, & 0=\mathbb{T}_{i(1)}^{\text{ \ }(j)},
& 0=\mathbb{T}_{(1)j}^{(i)}.%
\end{array}%
\end{array}%
\right.  \label{E-1}
\end{equation}
\end{proposition}

\begin{corollary}
The momentum stress-energy d-tensor of matter $\mathcal{T}$ satisfies the
following \textbf{geometrical conservation-like laws} (sum by $m$):%
\begin{equation*}
\left\{ 
\begin{array}{l}
\bigskip \mathbb{T}_{1/1}^{1}+\mathbb{T}_{1|m}^{m}+\mathbb{T}%
_{(m)1}^{(1)}|_{(1)}^{(m)}=0 \\ 
\bigskip \mathbb{T}_{i/1}^{1}+\mathbb{T}_{i|m}^{m}+\mathbb{T}%
_{(m)i}^{(1)}|_{(1)}^{(m)}=\mathsf{E}_{i|m}^{m} \\ 
\mathbb{T}_{\text{ }(1)/1}^{1(i)}+\mathbb{T}_{\text{ \ }(1)|m}^{m(i)}+%
\mathbb{T}_{(m)(1)}^{(1)(i)}|_{(1)}^{(m)}=\dfrac{e^{-2\sigma }\left[ 
\mathcal{P}^{1111}\right] ^{1/2}}{\mathcal{K}}\cdot \left[ \dfrac{\Sigma
_{11}}{p_{i}^{1}}+2\dfrac{\partial \Sigma _{11}}{\partial p_{i}^{1}}\right] ,%
\end{array}%
\right.
\end{equation*}%
where (sum by $r$):$\medskip $

$\bigskip \mathbb{T}_{1}^{1}\overset{def}{=}h^{11}\mathbb{T}_{11}=\mathcal{K}%
^{-1}\left\{ 2e^{-2\sigma }\left[ \mathcal{P}^{1111}\right] ^{1/2}\Sigma
_{11}+\dfrac{3}{4}h^{11}e^{2\sigma }\left[ \mathcal{P}^{1111}\right]
^{-1/2}\right\} ,$

$\bigskip \mathbb{T}_{1}^{m}\overset{def}{=}\overset{\ast }{g}\;^{\!\!mr}%
\mathbb{T}_{r1}=0,\quad\mathbb{T}_{(m)1}^{(1)}\overset{def}{=}h^{11}\overset{%
\ast }{g}_{mr}\mathbb{T}_{(1)1}^{(r)}=0,\quad \mathbb{T}_{i}^{1}\overset{def}%
{=}h^{11}\mathbb{T}_{1i}=0,$

$\bigskip \mathbb{T}_{i}^{m}\overset{def}{=}\overset{\ast }{g}\;^{\!\!mr}%
\mathbb{T}_{ri}=\mathsf{E}_{i}^{m}:=\mathcal{K}^{-1}\cdot \left[ \overset{%
\ast }{g}\;^{\!\!mr}R_{ri}+\delta _{i}^{m}\left\{ \frac{{}}{{}}2e^{-2\sigma }%
\left[ \mathcal{P}^{1111}\right] ^{1/2}\Sigma _{11}+\right. \right. $

$\bigskip \hspace{45mm}\left. \left. +\dfrac{3}{4}h^{11}e^{2\sigma }\left[ 
\mathcal{P}^{1111}\right] ^{-1/2}\right\} \right] ,$

$\bigskip \mathbb{T}_{(m)i}^{(1)}\overset{def}{=}h^{11}\overset{\ast }{g}%
_{mr}\mathbb{T}_{(1)i}^{(r)}=0,\;\;\mathbb{T}_{\text{ }(1)}^{1(i)}\overset{%
def}{=}h^{11}\mathbb{T}_{1(1)}^{\text{ \ }(i)}=0,\;\;\mathbb{T}_{\text{ \ }%
(1)}^{m(i)}\overset{def}{=}\overset{\ast }{g}\;^{\!\!mr}\mathbb{T}_{r(1)}^{%
\text{ \ }(i)}=0,$

$\medskip \mathbb{T}_{(m)(1)}^{(1)(i)}\overset{def}{=}h^{11}\overset{\ast }{g%
}\;_{mr}^{\!\!}\mathbb{T}_{(1)(1)}^{(r)(i)}=\dfrac{h^{11}e^{2\sigma }\left[ 
\mathcal{P}^{1111}\right] ^{-1/2}}{8\mathcal{K}}\dfrac{p_{m}}{p_{i}}+$

$\bigskip \hspace{25mm}+\delta _{m}^{i}\cdot \left[ \dfrac{h^{11}e^{2\sigma }%
\left[ \mathcal{P}^{1111}\right] ^{-1/2}}{4\mathcal{K}}+\dfrac{2e^{-2\sigma }%
\left[ \mathcal{P}^{1111}\right] ^{1/2}\Sigma _{11}}{\mathcal{K}}\right] ,$

and we also have (summation by $m$ and $r$, but no sum by $i$)$\bigskip $

$\bigskip \mathbb{T}_{1/1}^{1}\overset{def}{=}\dfrac{\delta \mathbb{T}%
_{1}^{1}}{\delta t}+\mathbb{T}_{1}^{1}$\textsc{k}$_{11}^{1}-\mathbb{T}%
_{1}^{1}\text{\textsc{k}}_{11}^{1}=\dfrac{\delta \mathbb{T}_{1}^{1}}{\delta t%
},\quad \mathbb{T}_{1|m}^{m}\overset{def}{=}\dfrac{\delta \mathbb{T}_{1}^{m}%
}{\delta x^{m}}+\mathbb{T}_{1}^{r}H_{rm}^{m},$

$\bigskip \mathbb{T}_{(m)1}^{(1)}|_{(1)}^{(m)}\overset{def}{=}\dfrac{%
\partial \mathbb{T}_{(m)1}^{(1)}}{\partial p_{m}^{1}}-\mathbb{T}%
_{(r)1}^{(1)}C_{m(1)}^{r(m)}=\dfrac{\partial \mathbb{T}_{(m)1}^{(1)}}{%
\partial p_{m}^{1}},$

$\bigskip \mathbb{T}_{i/1}^{1}\overset{def}{=}\dfrac{\delta \mathbb{T}%
_{i}^{1}}{\delta t}+\mathbb{T}_{i}^{1}\text{\textsc{k}}_{11}^{1}-\mathbb{T}%
_{r}^{1}A_{i1}^{r}=\dfrac{\delta \mathbb{T}_{i}^{1}}{\delta t}+\mathbb{T}%
_{i}^{1}\text{\textsc{k}}_{11}^{1},$

$\bigskip \mathbb{T}_{i|m}^{m}\overset{def}{=}\dfrac{\delta \mathbb{T}%
_{i}^{m}}{\delta x^{m}}+\mathbb{T}_{i}^{r}H_{rm}^{m}-\mathbb{T}%
_{r}^{m}H_{im}^{r}=\mathsf{E}_{i|m}^{m}:=\dfrac{\delta \mathsf{E}_{i}^{m}}{%
\delta x^{m}}+4\mathsf{E}_{i}^{m}\sigma _{m}-4\mathsf{E}_{i}^{i}\sigma _{i},$

$\bigskip \mathbb{T}_{(m)i}^{(1)}|_{(1)}^{(m)}\overset{def}{=}\dfrac{%
\partial \mathbb{T}_{(m)i}^{(1)}}{\partial p_{m}^{1}}-\mathbb{T}%
_{(r)i}^{(1)}C_{m(1)}^{r(m)}-\mathbb{T}_{(m)r}^{(1)}C_{i(1)}^{r(m)},$

$\bigskip \mathbb{T}_{\text{ }(1)/1}^{1(i)}\overset{def}{=}\dfrac{\delta 
\mathbb{T}_{\text{ }(1)}^{1(i)}}{\delta t}+\mathbb{T}_{\text{ }%
(1)}^{1(r)}A_{r1}^{i},\quad \mathbb{T}_{\text{ \ }(1)|m}^{m(i)}\overset{def}{%
=}\dfrac{\delta \mathbb{T}_{\text{ \ }(1)}^{m(i)}}{\delta x^{m}}+4\mathbb{T}%
_{\text{ \ }(1)}^{m(i)}\sigma _{m}+4\mathbb{T}_{\text{ }(1)}^{i(i)}\sigma
_{i},$

$\mathbb{T}_{(m)(1)}^{(1)(i)}|_{(1)}^{(m)}\overset{def}{=}\dfrac{\partial 
\mathbb{T}_{(m)(1)}^{(1)(i)}}{\partial p_{m}^{1}}-\mathbb{T}%
_{(r)(1)}^{(1)(i)}C_{m(1)}^{r(m)}+\mathbb{T}%
_{(m)(1)}^{(1)(r)}C_{r(1)}^{i(m)}.\medskip $
\end{corollary}

\begin{proof}
The local Einstein equations (\ref{E-1}), together with some direct
computations, lead us to what we were looking for.
\end{proof}

\subsection{Momentum electromagnetic-like geometrical model}

\hspace{5mm}In the paper \cite{Oana+Neag-3}, a geometrical theory for an
electromagnetism depending on momenta was also created, using only a given
Hamiltonian function $H$ on the momentum phase space $J^{1\ast }(\mathbb{R}%
,M^{4})$. In the background of the jet momentum Hamiltonian geometry from
this paper, we work with the \textit{electromagnetic distinguished }$2$%
\textit{-form} (sum by $i$ and $j$)%
\begin{equation*}
\mathbb{F}=F_{(1)j}^{(i)}\delta p_{i}^{1}\wedge dx^{j},
\end{equation*}%
where (sum by$\ r$ and $m$)%
\begin{equation*}
F_{(1)j}^{(i)}=\frac{h^{11}}{2}\left[ \overset{\ast }{g}\;^{\!\!jr}\underset{%
2}{N}\text{{}}_{(r)i}^{(1)}-\overset{\ast }{g}\;^{\!\!ir}\underset{2}{N}%
\text{{}}_{(r)j}^{(1)}+\left( \overset{\ast }{g}\;^{\!\!jr}H_{ri}^{m}-%
\overset{\ast }{g}\;^{\!\!ir}H_{rj}^{m}\right) p_{m}^{1}\right] .
\end{equation*}%
The above electromagnetic components depending on momenta are characterized
by some natural \textit{geometrical Maxwell-like equations} (for more
details, see Oan\u{a} and Neagu \cite{Oana+Neag-2}, \cite{Oana+Neag-3}).

By a direct calculation, we see that the $(t,x)$-conformal deformed
Berwald-Mo\'{o}r Hamiltonian metric of order four (\ref{rheon-B-M2})
produces null momentum electromagnetic components: $F_{(1)j}^{(i)}=0$.
Consequently, our dual jet $(t,x)$-conformal deformed Berwald-Mo\'{o}r
Hamiltonian geometrical electromagnetic theory is trivial one. Probably,
this fact suggests that the dual jet $(t,x)$-conformal deformed Berwald-Mo%
\'{o}r Hamiltonian structure (\ref{rheon-B-M2}) has rather gravitational
connotations than electromagnetic ones on the momentum phase space $J^{1\ast
}(\mathbb{R},M^{4})$.

\textbf{Acknowledgements.} The authors of this paper thank to Professors Gh.
Atanasiu and D.G. Pavlov for their encouragements and the useful discussions
on the Berwald-Mo\'{o}r research topic.

Alexandru \textsf{OAN\u{A}} and Mircea \textsf{NEAGU}

University Transilvania of Bra\c{s}ov, Department of Mathematics and
Informatics, Blvd. Iuliu Maniu, no. 50, Bra\c{s}ov 500091, Romania.

\textit{E-mails:} alexandru.oana@unitbv.ro, mircea.neagu@unitbv.ro

\end{document}